\documentclass[3p,number,sort&compress]{elsarticle}
\usepackage{hyperref}
\usepackage{color}
\usepackage{amssymb}
\usepackage{amsmath}

\newtheorem{theorem}{Theorem}[section]
\newtheorem{conjecture}[theorem]{Conjecture}
\newtheorem{lemma}[theorem]{Lemma}
\newtheorem{corollary}[theorem]{Corollary}
\newproof{proof}{Proof}
\numberwithin{equation}{section}

\newcommand{\C}{\mathbb{C}}
\newcommand{\F}{\mathbb{F}}
\newcommand{\N}{\mathbb{N}}
\newcommand{\R}{\mathbb{R}}
\newcommand{\eps}{\varepsilon}
\newcommand{\var}{\operatorname{Var}_p}

\begin{document}
\begin{frontmatter}
\title{Multiplication is an open bilinear mapping in the Banach algebra\\
of functions of bounded Wiener $p$-variation}

\author[TC]{Tiago Canarias}
\ead{t.canarias@campus.fct.unl.pt}

\author[AK]{Alexei Karlovich\corref{Alexei}}
\ead{oyk@fct.unl.pt}

\author[ES]{Eugene Shargorodsky}
\ead{eugene.shargorodsky@kcl.ac.uk}

\cortext[Alexei]{Corresponding author}

\address[TC]{
Departamento de Matem\'atica,
Faculdade de Ci\^encias e Tecnologia,\\
Universidade Nova de Lisboa,
Quinta da Torre,
2829--516 Caparica,
Portugal}

\address[AK]{
Centro de Matem\'atica e Aplica\c{c}\~oes,
Departamento de Matem\'atica,
Faculdade de Ci\^encias e Tecnologia,\\
Universidade Nova de Lisboa,
Quinta da Torre,
2829--516 Caparica,
Portugal}

\address[ES]{%
Department of Mathematics,
King's College London,
Strand, London WC2R 2LS,
United Kingdom}
\begin{abstract}
Let $BV_p[0,1]$, $1\le p<\infty$, be the Banach algebra of functions of 
bounded $p$-variation in the sense of Wiener. Recently, Kowalczyk and Turowska 
\cite{KT19} proved that the multiplication in $BV_1[0,1]$ is an open bilinear
mapping. We extend this result for all values of $p\in[1,\infty)$.
\end{abstract}
\end{frontmatter}
\section{Introduction}
Let $\mathcal{A}$ be a Banach algebra with a Banach algebra 
norm $\|\cdot\|_{\mathcal{A}}$.
We denote by $B_\mathcal{A}(a,\eps)$ the open ball in $\mathcal{A}$ centered
at $a$ of radius $\eps>0$, that is, 
\[
B_\mathcal{A}(a,\eps):=\{b\in\mathcal{A}:\|a-b\|_\mathcal{A}<\eps\}.
\]
We say that the multiplication in $\mathcal{A}$ is a bilinear mapping locally 
open at a pair $(a,b)\in\mathcal{A}^2:=\mathcal{A}\times\mathcal{A}$ if for 
every $\eps>0$ there exists a $\delta>0$ such that
\[
B_\mathcal{A}(a\cdot b,\delta)\subset 
B_\mathcal{A}(a,\eps)\cdot B_\mathcal{A}(b,\eps),
\]
where
\[
B_\mathcal{A}(a,\eps)\cdot B_\mathcal{A}(b,\eps)
:=
\{c\cdot d\in\mathcal{A}\ :\ 
c\in B_\mathcal{A}(a,\eps),\ d\in B_\mathcal{A}(b,\eps)\}.
\]
Following \cite{KT19}, the multiplications in $\mathcal{A}$ is called
an open bilinear mapping if it is locally open at every pair 
$(a,b)\in\mathcal{A}^2$.

Note that the multiplication might not be an open bilinear mapping even in 
very simple situations. For instance, if $\mathcal{A}=C[0,1]$ is the algebra
of real continuous functions with the supremum norm
\begin{equation}\label{eq:supremum-norm}
\|f\|_\infty:=\sup_{x\in[0,1]}|f(x)|,
\end{equation}
then for the function $g=x-1/2$ one has 
$g^2\in (B_\mathcal{A}(g,1/2))^2\setminus
\operatorname{int}\big((B_\mathcal{A}(g,1/2))^2\big)$,
where $\operatorname{int}(S)$ denotes the interior of a set $S$ 
(see \cite{BWW05}). Thus, the multiplication is not an open bilinear mapping 
in the algebra $C[0,1]$. This result was extended in \cite{W08} to the case
of the algebra $C^n[0,1]$ of $n$ times continuously differentiable functions.

The aim of this paper is to show that the multiplication is an open biliniear
mapping in the Banach algebra $BV_p[0,1]$, $1\le p<\infty$, of functions of 
bounded Wiener $p$-variation, extending the recent result by Kowlaczyk and Turowska
\cite{KT19} for $p=1$ to all values $p\in[1,\infty)$.

Let us recall the definition of functions of bounded Wiener $p$-variation. 
Suppose that $0\le\alpha\le\beta\le 1$.  
Let $\mathcal{P}[\alpha,\beta]$ be the set of all
partitions $P=\{t_0,\dots,t_m\}$ of the segment $[\alpha,\beta]$ of the 
form 
\[
\alpha=t_0<t_1<\dots<t_m=\beta. 
\]
Following \cite{W24} and 
\cite[Definition~1.31]{ABM14}, for a given a real number $p\in[1,\infty)$,
a partition $P=\{t_0,\dots,t_m\}\in\mathcal{P}[\alpha,\beta]$ and
a function $f:[\alpha,\beta]\to\F\in\{\R,\C\}$, the nonnegative number
\[
\var(f,P,[\alpha,\beta]):=\sum_{j=1}^{m}|f(t_j)-f(t_{j-1})|^p
\]
is called the Wiener $p$-variation of $f$ on $[\alpha,\beta]$ with respect to
$P$, while the (possibly infinite) number
\[
\var(f,[\alpha,\beta]):=
\sup\{\var(f,P,[\alpha,\beta]):P\in\mathcal{P}[\alpha,\beta]\},
\]
where the supremum is taken over all partitions of $[\alpha,\beta]$,
is called the total Wiener $p$-variation of $f$
on $[\alpha,\beta]$. Let
\[
BV_p[0,1]:=\{f:[0,1]\to \F\in\{\R,\C\} : \var(f,[0,1]<\infty\}
\]
be the set of all functions of bounded Wiener $p$-variation.
It is well known that $BV_p[0,1]$ is a Banach algebra
with respect to the pointwise multiplication and the norm
\begin{equation}\label{eq:BVp-norm}
\|f\|_{BV_p}:=\|f\|_\infty+\big(\var(f,[0,1])\big)^{1/p},
\end{equation}
where $\|f\|_\infty$ is given by \eqref{eq:supremum-norm}
(for instance, this result follows from \cite[Theorem~3.7 and 
Corollary~3.8]{DN11} with $\Phi(t)=t^p$, $1\le p<\infty$).
\begin{theorem}[Main result]
\label{th:main}
Let $1\le p<\infty$. Then the multiplication in the Banach algebra $BV_p[0,1]$
is an open bilinear mapping.
\end{theorem}

The paper is organized as follows. In Section~\ref{sec:local-openness},
following the main lines of the proof of \cite[Theorem~2.4]{KT19},
we show that the multiplication in a Banach algebra continuously 
embedded into the Banach algebra $B[0,1]$ of bounded functions 
and satisfying natural assumptions (the so-called symmetry property, the 
inverse closedness property and the selection principle) is locally open at 
every pair  of functions $(F,G)$ such that $|F|+|G|$ is bounded away from 
zero. We call such functions $F$ and $G$ jointly nondegenerate. 
Further, we show that the Banach algebra $BV_p[0,1]$ of functions
of bounded $p$-variation in the Wiener sense and the Banach algebra 
$\Lambda_p BV[0,1]$ of functions of bounded variation in the 
Shiba-Waterman sense (see \cite{HLP11,S80,W72}) satisfy
the hypotheses of the above result. In Section~\ref{sec:key-lemma},
we extend \cite[Lemma~2.1]{KT19} from the setting of $BV_1[0,1]$
to the setting of $BV_p[0,1]$ with an arbitrary $p\ge 1$. We should note that
the passage from $p=1$ to an arbitrary $p\ge 1$ is not trivial. 
In Section~\ref{sec:approximation-nondegenerated}, with the aid
of the main result of Section~\ref{sec:key-lemma} and following the scheme 
of the proof of \cite[Theorem~2.2]{KT19}, we show that an arbitrary
pair of functions $(F,G)\in (BV_p[0,1])^2$ can be approximated
by a pair of jointly nondegenerate functions $(F_1,G_1)\in (BV_p[0,1])^2$
such that $F\cdot G=F_1\cdot G_1$. In Section~\ref{sec:proof-conjecture}, 
we prove Theorem~\ref{th:main} combining the results of 
Sections~\ref{sec:local-openness} and \ref{sec:approximation-nondegenerated}. 
We conclude the paper with the conjecture that multiplication is an open 
bilinear mapping also in the  Banach algebra $\Lambda_p BV[0,1]$ of 
functions of bounded variation in the sense of Shiba-Waterman. 

This work started as an Undergraduate Research Opportunity Project of the 
first author at NOVA University of Lisbon in January-February of 2020
under the supervision of the second author.
\section{Local openness of multiplication in algebras of bounded functions}
\label{sec:local-openness}
Let $B[0,1]$ denote the Banach algebra of all bounded functions 
$f:[0,1]\to\F$, where $\F\in\{\R,\C\}$, with the norm given by 
\eqref{eq:supremum-norm}. We say that functions $f,g\in B[0,1]$ 
are jointly nondegenerate  if
\[
\inf_{x\in[0,1]}\big(|f(x)|+|g(x)|\big)>0.
\]

Let $\mathcal{F}[0,1]$ be a Banach algebra equipped with a norm 
$\|\cdot\|_{\mathcal{F}}$ and continuously embedded into the algebra 
$B[0,1]$. 
We will say that the algebra $\mathcal{F}[0,1]$ satisfies the symmetry 
property if for every function $f\in\mathcal{F}[0,1]$, its complex conjugate 
$\overline{f}$ also belongs to $\mathcal{F}[0,1]$ and 
$\| \overline{f} \|_{\mathcal{F}} = \| f \|_{\mathcal{F}}$. It is clear 
that every real algebra $\mathcal{F}[0,1]$ has the symmetry property.

Further, we will say that $\mathcal{F}[0,1]$ satisfies the inverse 
closedness property if for every $f\in\mathcal{F}[0,1]$, the inequality
\[
\inf_{x\in[0,1]}|f(x)|>0
\]
implies that $1/f\in\mathcal{F}[0,1]$ and
\[
\left\|\frac{1}{f}\right\|_{\mathcal{F}}
\le
\left(\inf_{x\in[0,1]}|f(x)|\right)^{-2}
\|f\|_{\mathcal{F}}.
\]

Finally, we will say that $\mathcal{F}[0,1]$ satisfies the selection 
principle if from every sequence of functions $\{f_n\}$ satisfying 
\[
\sup_{n\in\N}\|f_n\|_{\mathcal{F}}<\infty
\]
one can extract a subsequence $\{f_{n_k}\}$ that converges pointwise on 
$[0,1]$ to a function $f\in\mathcal{F}[0,1]$.
\begin{theorem}\label{th:local}
Let $\mathcal{F}[0,1]$ be a Banach algebra continuously embedded into 
the Banach algebra $B[0,1]$. Suppose that the algebra $\mathcal{F}[0,1]$ 
satisfies the symmetry property, the inverse closedness property and the
selection principle. Then the multiplication in $\mathcal{F}[0,1]$ is locally 
open at every pair of jointly nondegenerate functions 
$(F,G)\in (\mathcal{F}[0,1])^2$. 
\end{theorem}
\begin{proof}
The proof is analogous to that of \cite[Theorem~2.4]{KT19}. 
Since $\mathcal{F}[0,1]$ is continuously embedded into $B[0,1]$,
there is a constant $C\ge 1$ such that for all $f\in\mathcal{F}[0,1]$,
\begin{equation}\label{eq:local-0}
\sup_{x\in[0,1]}|f(x)|\le C\|f\|_{\mathcal{F}}.
\end{equation}
Without loss of generality, we can suppose that $\eps\in(0,1)$. Take
\begin{equation}\label{eq:local-1}
\delta:=\min\left\{1,\frac{1}{2}\inf_{x\in[0,1]}
\big(|F(x)|+|G(x)|\big)\right\}
\end{equation}
and
\begin{equation}\label{eq:local-2}
K:=2\max\big\{\|F\|_\mathcal{F},\|G\|_{\mathcal{F}},1\big\}.
\end{equation}
Let $h\in\mathcal{F}[0,1]$ be such that
\begin{equation}\label{eq:local-3}
\|h\|_{\mathcal{F}}<\eps\cdot \frac{\delta^8}{128 C K^6}.
\end{equation}
Consider
\begin{equation}\label{eq:local-4}
F_0:=F,\quad G_0:=G,\quad h_0:=h
\end{equation}
and define sequences $\{F_n\}_{n=0}^\infty$, $\{G_n\}_{n=0}^\infty$, and
$\{h_n\}_{n=0}^\infty$ inductively by
\begin{align}
F_{n+1} &:= F_n+h_n\cdot\frac{\overline{G_n}}{|F_n|^2+|G_n|^2},
\label{eq:local-5}
\\
G_{n+1} &:= G_n+h_n\cdot\frac{\overline{F_n}}{|F_n|^2+|G_n|^2},
\label{eq:local-6}
\\
h_{n+1}&:=-h_n^2\cdot\frac{\overline{F_n G_n}}{(|F_n|^2+|G_n|^2)^2}.
\label{eq:local-7}
\end{align}

We claim that for $n\in\N\cup\{0\}$,
\begin{enumerate}
\item[(i)]
\[
F_nG_n+h_n=FG+h,
\]
\item[(ii)] 
\[
\|F_n\|_{\mathcal{F}}\le\frac{K}{2}+1-2^{-n},
\quad
\|G_n\|_{\mathcal{F}}\le\frac{K}{2}+1-2^{-n},
\]
\item[(iii)]
\[
\inf_{x\in[0,1]}\big(|F_n(x)|+|G_n(x)|\big)\ge\delta+\delta\cdot 2^{-n},
\]
\item[(iv)]
\[
\|h_n\|_{\mathcal{F}}\le\eps\cdot 2^{-n}\cdot\frac{\delta^8}{128C K^6}.
\]
\end{enumerate}

We will prove these claims by induction. It follows from 
\eqref{eq:local-4} that
\[
F_0G_0+h_0=FG+h.
\]
We obtain from \eqref{eq:local-1}--\eqref{eq:local-4} that
\[
\|F_0\|_{\mathcal{F}}=\|F\|_{\mathcal{F}}\le\frac{K}{2},
\quad
\|G_0\|_{\mathcal{F}}=\|G\|_{\mathcal{F}}\le\frac{K}{2},
\quad
\|h_0\|_{\mathcal{F}}=\|h\|_{\mathcal{F}}<\eps\cdot\frac{\delta^8}{128CK^6},
\]
\[
\inf_{x\in[0,1]}\big(|F_0(x)|+|G_0(x)|\big)
=
\inf_{x\in[0,1]}\big(|F(x)|+|G(x)|\big)\ge 2\delta.
\]
That is, (i)--(iv) are satisfied for $n=0$.

Now we assume that (i)--(iv) are fulfilled for some $n\in\N\cup\{0\}$. 
Then, taking into account \eqref{eq:local-2}, we see that $K/2\ge 1$ and
\begin{align}
&
F_nG_n+h_n=FG+h,
\label{eq:local-8}
\\
& 
\|F_n\|_{\mathcal{F}}\le\frac{K}{2}+1-2^{-n}<K,
\label{eq:local-9}
\\
& 
\|G_n\|_{\mathcal{F}}\le\frac{K}{2}+1-2^{-n}<K,
\label{eq:local-10}
\\
&\inf_{x\in[0,1]}\big(|F_n(x)|+|G_n(x)|\big)\ge\delta+\delta\cdot 2^{-n}>\delta,
\label{eq:local-11}
\\
&
\|h_n\|_{\mathcal{F}}\le \eps\cdot 2^{-n}\cdot\frac{\delta^8}{128CK^6}.
\label{eq:local-12}
\end{align}
Let us show that (i)--(iv) are fulfilled for $n+1$.

(i) It follows from \eqref{eq:local-5}--\eqref{eq:local-8} that
\begin{align*}
F_{n+1}G_{n+1}+h_{n+1}
&=
\left(F_n+\frac{h_n\cdot\overline{G_n}}{|F_n|^2+|G_n|^2}\right)
\left(G_n+\frac{h_n\cdot\overline{F_n}}{|F_n|^2+|G_n|^2}\right)
-
\frac{h_n^2\cdot\overline{F_nG_n}}{(|F_n|^2+|G_n|^2)^2}
\\
&= 
F_nG_n+h_n\frac{F_n\overline{F_n}+G_n\overline{G_n}}{|F_n|^2+|G_n|^2}
+
h_n^2\frac{\overline{F_nG_n}}{(|F_n|^2+|G_n|^2)^2}
-
h_n^2\frac{\overline{F_nG_n}}{(|F_n|^2+|G_n|^2)^2}
\\
&= 
F_nG_n+h_n=FG+h.
\end{align*}
Hence, (i) is satisfied for $n+1$.

(ii) Since $\mathcal{F}[0,1]$ is a Banach algebra 
satisfying the symmetry property, we obtain from
\eqref{eq:local-9} and \eqref{eq:local-10} that
\begin{align}
\| |F_n|^2+|G_n|^2\|_{\mathcal{F}}
&\le 
\|F_n\cdot \overline{F_n}\|_{\mathcal{F}}
+
\|G_n\cdot\overline{G_n}\|_{\mathcal{F}}
\le
\| F_n\|_{\mathcal{F}}
\|\overline{F_n}\|_{\mathcal{F}}
+
\|G_n\|_{\mathcal{F}}\|\overline{G_n}\|_{\mathcal{F}}
\nonumber\\
&= 
\|F_n\|_{\mathcal{F}}^2 +\|G_n\|_{\mathcal{F}}^2
\le
K^2+K^2=2K^2.
\label{eq:local-13}
\end{align}
It follows from \eqref{eq:local-11} that for every $x\in[0,1]$,
\[
\delta^2 
\le 
\big(|F_n(x)|+|G_n(x)|\big)^2
=
|F_n(x)|^2+2|F_n(x)|\cdot|G_n(x)|+|G_n(x)|^2
\le 
2\big(|F_n(x)|^2+|G_n(x)|^2\big).
\]
Hence
\begin{equation}\label{eq:local-14}
\inf_{x\in[0,1]}\big(|F_n(x)|^2+|G_n(x)|^2\big)\ge\frac{\delta^2}{2}.
\end{equation}
Taking into account that $\mathcal{F}[0,1]$ is a Banach algebra
with the symmetry property, it follows from
\eqref{eq:local-5} and \eqref{eq:local-9}--\eqref{eq:local-10} that
\begin{align}
\|F_{n+1}\|_{\mathcal{F}}
&\le 
\|F_n\|_{\mathcal{F}}+\|h_n\|_{\mathcal{F}}\|G_n\|_{\mathcal{F}}
\left\|\frac{1}{|F_n|^2+|G_n|^2}\right\|_{\mathcal{F}}
\nonumber\\
&\le 
\left(\frac{K}{2}+1-2^{-n}\right)+\|h_n\|_{\mathcal{F}}K
\left\|\frac{1}{|F_n|^2+|G_n|^2}\right\|_{\mathcal{F}}.
\label{eq:local-15}
\end{align}
Since $\mathcal{F}[0,1]$ has the inverse closedness property,
we deduce from \eqref{eq:local-13}--\eqref{eq:local-14} that
\begin{align}
\left\|\frac{1}{|F_n|^2+|G_n|^2}\right\|_{\mathcal{F}} 
\le
\left(\inf_{x\in[0,1]}\big(|F_n(x)|^2+|G_n(x)|^2\big)\right)^{-2}
\||F_n|^2+|G_n|^2\|_{\mathcal{F}} 
\le \left(\frac{2}{\delta^2}\right)^2 2K^2
=
\frac{8K^2}{\delta^4}.
\label{eq:local-16}
\end{align}
Combining \eqref{eq:local-15}--\eqref{eq:local-16} with \eqref{eq:local-12}
and taking into account that $\eps\in(0,1)$ and $C\ge 1$, we obtain
\begin{align}
\|F_{n+1}\|_{\mathcal{F}}
\le 
\frac{K}{2}+1-2^{-n}+\frac{8K^3}{\delta^4}
\cdot\eps\cdot 2^{-n}\cdot\frac{\delta^8}{128CK^6}
<
\frac{K}{2}+1-2^{-n}+2^{-n}\frac{\delta^4}{16K^3}.
\label{eq:local-17}
\end{align}
It follow from \eqref{eq:local-1}--\eqref{eq:local-2} that 
$\delta\le 1\le K/2$. Therefore
\begin{equation}\label{eq:local-18}
\frac{\delta^4}{16K^3} = \frac{\delta}{16} \left(\frac{\delta}{K}\right)^3 
\le 
\frac{\delta}{16}\cdot \frac{1}{8} 
= 
\frac{\delta}{128}<\frac{1}{2}.
\end{equation}
In view of \eqref{eq:local-17}--\eqref{eq:local-18} we obtain
\[
\|F_{n+1}\|_{\mathcal{F}}<\frac{K}{2}+1-2^{-n}+2^{-n-1}
=
\frac{K}{2}+1-2^{-n-1}.
\]
Analogously it can be shown that
\[
\|G_{n+1}\|_{\mathcal{F}}
<
\frac{K}{2}+1-2^{-n-1}.
\]
Thus, (ii) is fulfilled for $n+1$.

(iii) Since $\mathcal{F}[0,1]$ is a Banach algebra and $\eps\in(0,1)$,
it follows from \eqref{eq:local-5}, \eqref{eq:local-0},
\eqref{eq:local-10}, \eqref{eq:local-12}, \eqref{eq:local-16}, and
\eqref{eq:local-18} that for $x\in[0,1]$,
\begin{align*}
|F_n(x)|
&\le
|F_{n+1}(x)|+|h_n(x)|\frac{|G_n(x)|}{|F_n(x)|^2+|G_n(x)|^2}
\le 
|F_{n+1}(x)|+ C\|h_n\|_{\mathcal{F}}\|G_n\|_{\mathcal{F}}
\left\|\frac{1}{|F_n|^2+|G_n|^2}\right\|_{\mathcal{F}}
\\
&\le 
|F_{n+1}(x)|+
C\eps\cdot 2^{-n}\frac{\delta^8}{128CK^6}\cdot K\cdot
\frac{8K^2}{\delta^4}
<
|F_{n+1}(x)|+2^{-n}\cdot\frac{\delta^4}{16K^3}
<
|F_{n+1}(x)|+2^{-n}\cdot\frac{\delta}{4}.
\end{align*}
Hence
\begin{equation}\label{eq:local-19}
|F_{n+1}(x)|>|F_n(x)|-2^{-n-2}\delta,
\quad x\in[0,1].
\end{equation}
Analogously,
\begin{equation}\label{eq:local-20}
|G_{n+1}(x)|>|F_n(x)|-2^{-n-2}\delta,
\quad x\in[0,1].
\end{equation}
We conclude from \eqref{eq:local-11} and 
\eqref{eq:local-19}--\eqref{eq:local-20} that 
\begin{align*}
\inf_{x\in[0,1]}\big(|F_{n+1}(x)|+|G_{n+1}(x)|\big)
&\ge 
\inf_{x\in[0,1]}\big(|F_{n}(x)|+|G_{n}(x)|\big)
-2\cdot 2^{-n-2}\delta 
\\
&\ge 
\delta+\delta\cdot 2^{-n}-\delta\cdot 2^{-n-1}
=\delta+\delta\cdot 2^{-n-1}.
\end{align*}
Hence (iii) is fulfilled for $n+1$.

(iv) Since $\mathcal{F}[0,1]$ is a Banach algebra with the symmetry 
property, $\eps\in(0,1)$ and $C\ge 1$, it follows from 
\eqref{eq:local-7}, \eqref{eq:local-9}--\eqref{eq:local-10},
\eqref{eq:local-12} and \eqref{eq:local-16} that
\begin{align*}
\|h_{n+1}\|_{\mathcal{F}}
&\le 
\|h_n\|_{\mathcal{F}}^2
\|\overline{F_n}\|_{\mathcal{F}}
\|\overline{G_n}\|_{\mathcal{F}}
\left\|\frac{1}{|F_n|^2+|G_n|^2}\right\|_{\mathcal{F}}^2
=
\|h_n\|_{\mathcal{F}}^2
\|F_n\|_{\mathcal{F}}
\|G_n\|_{\mathcal{F}}
\left\|\frac{1}{|F_n|^2+|G_n|^2}\right\|_{\mathcal{F}}^2
\\
&\le
\left(\eps\cdot 2^{-n}\cdot\frac{\delta^8}{128CK^6}\right)^2
K^2
\left(\frac{8K^2}{\delta^4}\right)^2
=
\eps^2\cdot 2^{-2n-1}\cdot\frac{\delta^8}{128 C^2 K^6}
<
\eps\cdot 2^{-n-1}\cdot\frac{\delta^8}{128 C K^6}.
\end{align*}
Hence (iv) is fulfilled for $n+1$.

Thus, we have verified properties (i)--(iv) by induction for all 
$n\in\N\cup\{0\}$.

In view of (ii), the terms of the sequences $\{F_n\}_{n=0}^\infty$ 
and $\{G_n\}_{n=0}^\infty$ have uniformly bounded norms. By the selection
principle, there exist a subsequence $\{F_{n_k}\}_{k=0}^\infty$
of $\{F_n\}_{n=0}^\infty$ and a subsequence $\{G_{n_k}\}_{k=0}^\infty$
of $\{G_n\}_{n=0}^\infty$ such that for every $x\in[0,1]$,
\begin{equation}\label{eq:local-21}
\lim_{k\to\infty}F_{n_k}(x)=f(x),
\quad
\lim_{k\to\infty}G_{n_k}(x)=g(x),
\end{equation}
where $f,g\in\mathcal{F}[0,1]$.
It follows from \eqref{eq:local-0} and (iv) that for all $x\in[0,1]$,
\begin{equation}\label{eq:local-22}
\lim_{n\to\infty}|h_n(x)| 
\le
C\lim_{n\to\infty}\|h_n\|_{\mathcal{F}}
\le 
\frac{\eps\delta^8}{128CK^6}\lim_{n\to\infty}2^{-n}=0.
\end{equation}
In view of (i) and \eqref{eq:local-21}--\eqref{eq:local-22}, we obtain for 
$x\in[0,1]$,
\begin{align}
f(x)g(x)
=
\lim_{k\to\infty}F_{n_k}(x)G_{n_k}(x)
=
\lim_{k\to\infty}\big(F_{n_k}(x)G_{n_k}(x)+h_{n_k}(x)\big)
=F(x)G(x)+h(x).
\label{eq:local-23}
\end{align}
Since
\[
f(x)-F(x)
=
\lim_{k\to\infty}(F_{n_k}(x)-F(x))
=
\lim_{k\to\infty}\sum_{j=0}^{n_k}(F_{j+1}(x)-F_j(x))
=
\sum_{n=0}^\infty (F_{n+1}(x)-F_n(x)),
\]
$\mathcal{F}[0,1]$ is a Banach algebra with the symmetry property, 
$\eps\in(0,1)$ and $C\ge 1$,
we obtain from
\eqref{eq:local-5}, \eqref{eq:local-10}, \eqref{eq:local-12},
\eqref{eq:local-16},
and \eqref{eq:local-18}
that
\begin{align}
\|f-F\|_{\mathcal{F}}
&\le 
\sum_{n=0}^\infty \|F_{n+1}-F_n\|_{\mathcal{F}}
\le 
\sum_{n=0}^\infty 
\|h_n\|_{\mathcal{F}}
\|G_n\|_{\mathcal{F}}
\left\|\frac{1}{|F_n|^2+|G_n|^2}\right\|_{\mathcal{F}}
\nonumber\\
&\le 
\sum_{n=0}^\infty \eps\cdot 2^{-n}\cdot\frac{\delta^8}{128CK^6}
\cdot K\cdot\frac{8K^2}{\delta^4}
=
\frac{\eps}{C}\cdot\frac{\delta^4}{16K^3}\sum_{n=0}^\infty 2^{-n}
<
\eps.
\label{eq:local-24}
\end{align}
Analogously we can show that
\begin{equation}\label{eq:local-25}
\|g-G\|_{\mathcal{F}}<\eps.
\end{equation}
So, for every $h \in \mathcal{F}[0,1]$ satisfying \eqref{eq:local-3}, there 
exist $f$ and $g$ in $\mathcal{F}[0,1]$ such that \eqref{eq:local-24} and 
\eqref{eq:local-25} hold, and $FG + h = fg$ (see \eqref{eq:local-23}). 
This means that
\[
B_{\mathcal{F}[0,1]}(F\cdot G,\eta)
\subset 
B_{\mathcal{F}[0,1]}(F,\eps)\cdot B_{\mathcal{F}[0,1]}(G,\eps)
\]
with $\eta:=\eps\cdot\frac{\delta^8}{128CK^6}$.
Hence, the multiplication in the Banach algebra $\mathcal{F}[0,1]$
is locally open at the pair $(F,G)\in (\mathcal{F}[0,1])^2$.
\qed
\end{proof}
\begin{corollary}\label{co:local-BVp}
Let $1\le p<\infty$. Then the  multiplication in $BV_p[0,1]$ is locally open 
at every pair of jointly nondegenerate  functions $(F,G)\in (BV_p[0,1])^2$. 
\end{corollary}
\begin{proof}
We have to verify the hypotheses of Theorem~\ref{th:local}.
The definitions of the norms  \eqref{eq:BVp-norm} and \eqref{eq:supremum-norm}
immediately imply that the Banach algebra $BV_p[0,1]$ is continuously 
embedded into the Banach algebra $B[0,1]$ (with the embedding constant $1$)
and that the algebra $BV_p[0,1]$ satisfies the symmetry property.
It follows from the Helly-type selection theorem \cite[Theorem~2.49]{ABM14}
with $\Phi(t)=t^p$, $1\le p<\infty$, that $BV_p[0,1]$ satisfies the selection
principle. 

Let us show that $BV_p[0,1]$ has the inverse closedness property.
Take a function $f\in BV_p[0,1]$ such that
\begin{equation}\label{eq:local-BVp-1}
\inf_{x\in[0,1]}|f(x)|>0
\end{equation}
and a partition $P=\{t_0,\dots,t_m\}\in\mathcal{P}[0,1]$. Then $f(t_j)\ne 0$ for
$j\in\{0,\dots,m\}$ in view of \eqref{eq:local-BVp-1} and
\begin{align*}
\var(1/f,P,[0,1])
&=
\sum_{j=1}^m\left|\frac{1}{f(t_j)}-\frac{1}{f(t_{j-1})}\right|^p
=
\sum_{j=1}^m\left|\frac{f(t_j)-f(t_{j-1})}{f(t_j)f(t_j)}\right|^p
\\
&\le 
\left(\inf_{x\in[0,1]}|f(x)|\right)^{-2p}\var(f,P,[0,1]).
\end{align*}
Therefore 
\begin{equation}\label{eq:local-BVp-2}
\var(1/f,[0,1])\le\left(\inf_{x\in[0,1]}|f(x)|\right)^{-2p}\var(f,[0,1]).
\end{equation}
On the other hand,
\begin{equation}\label{eq:local-BVp-3}
\|1/f\|_\infty=\sup_{x\in[0,1]}|1/f(x)|=\left(\inf_{x\in[0,1]}|f(x)|\right)^{-1}.
\end{equation}
Combining \eqref{eq:local-BVp-2} and \eqref{eq:local-BVp-3}, we arrive at the 
following:
\begin{align}
\|1/f\|_{BV_p}
&=
\|1/f\|_\infty+\big(\var(1/f,[0,1])\big)^{1/p}
\nonumber\\
&\le 
\left(\inf_{x\in[0,1]}|f(x)|\right)^{-1}+
\left(\inf_{x\in[0,1]}|f(x)|\right)^{-2}\big(\var(f,[0,1]\big)^{1/p}
\nonumber\\
&\le  
\left(\inf_{x\in[0,1]}|f(x)|\right)^{-2}
\left(\|f\|_\infty+\big(\var(f,[0,1]\big)^{1/p}\right)
\nonumber\\
&=
\left(\inf_{x\in[0,1]}|f(x)|\right)^{-2}\|f\|_{BV_p}.
\label{eq:local-BVp-4}
\end{align}
Thus $BV_p[0,1]$ satisfies the inverse closedness property. It remains to
apply Theorem~\ref{th:local}.
\qed
\end{proof}

Let us show that the hypotheses of Theorem~\ref{th:local} are also 
satisfied in the case of Banach algebras of functions of generalized 
variation in the Shiba-Waterman sense. Shiba \cite{S80} introduced the 
class $\Lambda_p BV[0,1]$ with $1\le p<\infty$, extending the concept of 
the bounded $\Lambda$-variation in the sense of Waterman \cite{W72}. 
Let $\Lambda=\{\lambda_i\}_{i=1}^\infty$ be a nondecreasing sequence of 
positive numbers such that 
$\sum_{i=1}^\infty\frac{1}{\lambda_i}=+\infty$ and let $1\le p<\infty$.
A function $f:[0,1]\to\F\in\{\R,\C\}$ is said to be of bounded 
$\Lambda_p$-variation in the Shiba-Waterman sense if 
\[
\operatorname{Vap}_{\Lambda_p}(f,[0,1])
:=
\sup\sum_{i=1}^n\frac{|f(I_i)|^p}{\lambda_i}<+\infty,
\]
where the supremum is taken over all finite families $\{I_i\}_{i=1}^n$
of nonoverlapping intervals on $[0,1]$ and
$f(I_i):=f(\sup I_i)-f(\inf I_i)$. Let $\Lambda_pBV[0,1]$ be the set of 
all functions $f:[0,1]\to\F\in\{\R,\C\}$ of bounded $\Lambda_p$-variation.
Kantorowitz \cite[Theorem~1]{K11} proved 
that $\Lambda_p BV[0,1]$ is a Banach algebra with respect to the pontwise
multiplication and the norm
\begin{equation}\label{eq:Shiba-Waterman}
\|f\|_{\Lambda_p BV}:=\|f\|_\infty+ 
\big(\operatorname{Vap}_{\Lambda_p}(f,[0,1])\big)^{1/p}.
\end{equation}
\begin{corollary}\label{co:local-Shiba-Waterman}
Let $1\le p<\infty$. Then the  multiplication in $\Lambda_p BV[0,1]$ is 
locally open at every pair of jointly nondegenerate  functions 
$(F,G)\in (\Lambda_p BV[0,1])^2$. 
\end{corollary}
\begin{proof}
As in the proof of the previous corollary, we have to verify the hypotheses 
of Theorem~\ref{th:local}.
The definitions of the norms  \eqref{eq:Shiba-Waterman} and 
\eqref{eq:supremum-norm} immediately imply that the Banach algebra 
$\Lambda_p BV[0,1]$ is continuously  embedded into the Banach algebra 
$B[0,1]$ (with the embedding constant $1$) and that the algebra 
$\Lambda_p BV[0,1]$ satisfies the symmetry property. The selection principle
for the algebra $\Lambda_p BV[0,1]$ is proved in \cite[Theorem~3.2]{HLP11}.

If $f\in\Lambda_p BV[0,1]$ satisifies \eqref{eq:local-BVp-1}, then
for every interval $I\subset[0,1]$,
\[
|(1/f)(I)|\le\left(\inf_{x\in[0,1]}|f(x)|\right)^{-2}|f(I)|.
\]
Therefore 
\begin{equation}\label{eq:local-Shiba-Waterman-1}
\operatorname{Var}_{\Lambda,p}(1/f,[0,1])
\le
\left(\inf_{x\in[0,1]}|f(x)|\right)^{-2p}
\operatorname{Vap}_{\Lambda_p}(f,[0,1]).
\end{equation}
Combining \eqref{eq:local-Shiba-Waterman-1} and \eqref{eq:local-BVp-3}, 
similarly to \eqref{eq:local-BVp-4}, we obtain
\[
\|1/f\|_{\Lambda_p BV}
\le
\left(\inf_{x\in[0,1]}|f(x)|\right)^{-2}\|f\|_{\Lambda_p BV}.
\]
Thus $\Lambda_p BV[0,1]$ satisfies the inverse closedness property. 
It remains to apply Theorem~\ref{th:local}.
\qed
\end{proof}
\section{Key lemma}\label{sec:key-lemma}
The aim of this section is to prove an extension of \cite[Lemma~2.1]{KT19}
for the Banach algebras $BV_p[0,1]$ with arbitrary $p\in[1,\infty)$.

Let us start with several elementary inequalities.
\begin{lemma}\label{le:Shar-1}
Let $1 \le p < \infty$. Then
\begin{equation}\label{eq:Shar-1}
(1 + x)^p \le 1 + p2^{p -1} x  \quad\mbox{for all}\quad x \in [0, 1] .
\end{equation}
\end{lemma}
\begin{proof}
Integrating both sides of the inequality
\[
(1 + t)^{p -1} \le 2^{p -1} , \quad t \in [0, 1] 
\]
from $0$ to $x$, one gets
\[
\frac1p\left((1 + x)^p - 1\right) \le 2^{p -1} x ,
\]
which is equivalent to \eqref{eq:Shar-1}.
\qed
\end{proof}
\begin{lemma}\label{le:Shar-2}
Let $1 \le p < \infty$. Then
\begin{equation}\label{eq:Shar-2}
(a + b)^p \le a^p + \max\{p, 2\}\, 2^{p -1} b  
\quad\mbox{for all}\quad a, b \in [0, 1] .
\end{equation}
\end{lemma}
\begin{proof}
If $a = 0$ then \eqref{eq:Shar-2} holds because $b^p \le b$. Suppose $a > 0$. 
If $b \le a$, then it follows from Lemma~\ref{le:Shar-1} that
\begin{align}
(a + b)^p 
= 
a^p 
\left(1 +\frac{b}{a}\right)^p \le a^p \left(1 + p2^{p -1}\, \frac{b}{a}\right) 
= 
a^p + p2^{p -1} a^{p - 1} b \le a^p + p2^{p -1}  b . 
\label{eq:Shar-3}
\end{align}
If $b > a$, then
\begin{equation}\label{eq:Shar-4}
(a + b)^p < (2b)^p = 2^p b^p < a^p + 2^p b^p \le a^p + 2^p b .
\end{equation}
Estimate \eqref{eq:Shar-2} follows from \eqref{eq:Shar-3} and \eqref{eq:Shar-4}.
\qed
\end{proof}
\begin{corollary}\label{co:Shar}
Let $1 \le p < \infty$ and $u, v \in \mathbb{C}$ be such that 
$|u - v|, |v| \le 1$. Then
\begin{equation}\label{eq:Shar-5}
|u - v|^p \ge |u|^p - \max\{p, 2\}\, 2^{p -1} |v|.
\end{equation}
\end{corollary}
\begin{proof}
Using \eqref{eq:Shar-2} with $a = |u - v|$ and $b = |v|$, one gets
\[
|u|^p \le (|u - v| + |v|)^p \le |u - v|^p + \max\{p, 2\}\, 2^{p -1} |v|,
\]
which immediately implies \eqref{eq:Shar-5}.
\qed
\end{proof}

The following lemma is a special case of the desired result for functions
with values in the segment $[0,1]$.
\begin{lemma}\label{le:Shar-4}
Let $1 \le p < \infty$ and let $f \in BV_p[0, 1]$ be such that 
$f : [0, 1] \to [0, 1]$. For any $\varepsilon > 0$ there exist $\eta > 0$ such 
that if
\begin{equation}\label{eq:Shar-6}
0 \le x_1 < x_2 <\cdots < x_m \le 1 \quad\mbox{and}\quad 
f(x_j) < \eta , \ j = 1, \dots, m ,
\end{equation}
then
\begin{equation}\label{eq:Shar-7}
\left(\sum_{j = 1}^{m - 1} |f(x_{j + 1}) - f(x_j)|^p\right)^{1/p} < \varepsilon.
\end{equation}
\end{lemma}
\begin{proof}
Choose a partition $0 = y_1 < y_2 < \cdots < y_n = 1$ such that
\[
\sum_{k = 1}^{n- 1} |f(y_{k+ 1}) - f(y_k)|^p 
> 
\var(f,[0,1]) -\frac{\varepsilon^p}{2}.
\]
Set 
\[
\eta = \min\left\{1, \frac{\varepsilon^p}{n(p + 2)2^{p + 1}}\right\}.
\] 
Suppose \eqref{eq:Shar-6} holds. If $[y_k, y_{k + 1}]$ contains some of the 
points $x_1, \dots, x_m$, let
\[
j_k :=\min\{j : \ x_j \in [y_k, y_{k + 1}]\},
\quad
J_k :=\max\{j : \ x_j \in [y_k, y_{k + 1}]\}. 
\]
Note that since $f \ge 0$, one has
\[
(f(y_k))^p + (f(y_{k + 1}))^p 
\ge 
\left(\max\{f(y_k), f(y_{k + 1})\}\right)^p 
\ge 
|f(y_{k + 1}) - f(y_k)|^p .
\]
Then using Corollary \ref{co:Shar}, one gets
\begin{align*}
&
|f(x_{j_k}) 
- 
f(y_k)|^p + 
|f(x_{j_k + 1}) - f(x_{j_k})|^p + 
\cdots +  
|f(x_{J_k}) - f(x_{J_k - 1})|^p 
+ |f(y_{k + 1}) - f(x_{J_k})|^p 
\\
&\quad \ge 
(f(y_k))^p - \max\{p, 2\}\, 2^{p -1} f(x_{j_k}) + 
\sum_{j = j_k}^{J_k - 1} |f(x_{j + 1}) - f(x_j)|^p 
+ 
(f(y_{k + 1}))^p - \max\{p, 2\}\, 2^{p -1} f(x_{J_k}) 
\\
&\quad\ge  
|f(y_{k + 1}) - f(y_k)|^p - \max\{p, 2\}\, 2^p \eta + 
\sum_{j = j_k}^{J_k} |f(x_{j + 1}) - f(x_j)|^p - \eta^p 
\\
&\quad\ge 
|f(y_{k + 1}) - f(y_k)|^p - (p + 2)2^{p} \eta + 
\sum_{j = j_k}^{J_k} |f(x_{j + 1}) - f(x_j)|^p ,
\end{align*}
where we take $f(x_{m + 1}) = 0$ if $J_k = m$. In the last inequality 
above, we have used the following inequality
\[
\max\{p, 2\} + 1 \le p + 2.
\]
Summing over $k$ from $1$ to $n - 1$, one obtains
\begin{align*}
\var(f,[0,1]) 
\ge & 
\sum_{k = 1}^{n- 1} |f(y_{k+ 1}) - f(y_k)|^p - (n - 1) (p + 2)2^{p} \eta 
+ 
\sum_{j = 1}^{m - 1} |f(x_{j + 1}) - f(x_j)|^p 
\\
> & 
\var(f,[0,1]) - \frac{\varepsilon^p}{2} - \frac{\varepsilon^p}{2} + 
\sum_{j = 1}^{m - 1} |f(x_{j + 1}) - f(x_j)|^p ,
\end{align*}
which proves \eqref{eq:Shar-7}.
\qed
\end{proof}

We are now in a position to prove the main result of this section.
For $p=1$ the following lemma was proved in \cite[Lemma~2.1]{KT19}.
\begin{lemma}[Key lemma]
\label{le:Shar}
Let $1 \le p < \infty$ and $f \in BV_p[0, 1]$. For any $\eps > 0$ there exist 
$\delta > 0$ such that if
\[
0 \le x_1 < x_2 <\cdots < x_m \le 1 \quad\mbox{and}\quad 
|f(x_j)| < \delta \quad\mbox{for}\quad j \in\{1, \dots, m\} ,
\]
then
\[
\left(\sum_{j = 1}^{m - 1} |f(x_{j + 1}) - f(x_j)|^p\right)^{1/p} < \eps.
\]
\end{lemma}
\begin{proof}
There is nothing to prove if $f = 0$. So, we assume that $f \not=0$. 
Let $M := \|f\|_\infty$, $f_0 := \frac1M\, f$.
Let $u$ and $v$ be the real and the imaginary parts of $f_0$. Hence
$f_0 = u + iv$. Consider the functions
\[
w_1 = u_+ := \max\{u, 0\} = \frac{|u| + u}{2}, 
\quad
w_2 = u_- := (-u)_+ =  \frac{|u| - u}{2} = u_+ - u 
\]
and $w_3 = v_+$, $w_4 = v_-$.
Then $f_0 = w_1 - w_2 + i(w_3 - w_4)$ and 
\[
0 \le w_l \le \|f_0\|_\infty = 1  
\quad\mbox{for}\quad
l \in\{ 1, 2, 3, 4\}.
\]
Since $|a_+ - b_+| \le |a - b|$ for all $a, b \in \mathbb{R}$, one also has 
\[
\var(w_l,[0,1]) \le \var(f_0,[0,1]) = \frac{1}{M^p}\, \var(f,[0,1])
\quad\mbox{for}\quad
l \in\{ 1, 2, 3, 4\}.
\]
Take an arbitrary $\varepsilon > 0$. It follows from Lemma \ref{le:Shar-4} 
that for every $l\in\{ 1, 2, 3, 4\}$, there exists $\eta_l > 0$ such that
\[
0 \le x_1 < x_2 <\cdots < x_m \le 1 \quad\mbox{and}\quad w_l(x_j) < \eta_l , 
\quad 
j = 1, \dots, m
\]
imply
\[
\left(\sum_{j = 1}^{m - 1} |w_l(x_{j + 1}) - w_l(x_j)|^p\right)^{1/p} 
< 
\frac{\varepsilon}{4M}.
\]
Let $\eta := M \min\{\eta_l : \ l = 1, 2, 3, 4\}$. If
\[
0 \le x_1 < x_2 <\cdots < x_m \le 1 \quad\mbox{and}\quad |f(x_j)| < \eta , \ 
j = 1, \dots, m ,
\]
then
\[
w_l(x_j) < \frac1M\, \eta \le \eta_l , \ j = 1, \dots, m ,
\]
and it follows from the above that
\begin{align*}
\left(\sum_{j = 1}^{m - 1} |f(x_{j + 1}) - f(x_j)|^p\right)^{1/p} 
&= 
M \left(\sum_{j = 1}^{m - 1} |f_0(x_{j + 1}) - f_0(x_j)|^p\right)^{1/p} 
\\
&\le 
M \sum_{l = 1}^4 \left(
\sum_{j = 1}^{m - 1} |w_l(x_{j + 1}) - w_l(x_j)|^p
\right)^{1/p} 
<  
M \sum_{l = 1}^4 \frac{\varepsilon}{4M}= \varepsilon,
\end{align*}
which completes the proof.
\qed
\end{proof}
\section{Approximating in $BV_p[0,1]$ an arbitrary pair of functions 
by a pair of jointly nondegenerate  functions}
\label{sec:approximation-nondegenerated}
Let us start this section with two simple lemmas.
\begin{lemma}\label{le:discontinuities}
Let $1\le p<\infty$ and $f\in BV_p[0,1]$. Then $f$ possesses a limit from 
the left and from  the right at each point. Moreover $f$ has a most countably 
many discontinuities.
\end{lemma}

This statement can be proved as in the case $p=1$ (see, e.g., 
\cite[Proposition~1.32 and Corollary~1.33]{D98}).
\begin{lemma}\label{le:Karl}
Let $1\le p<\infty$, $\rho > 0$, and $f:(a,b)\to\C$ be such that
\[
\inf_{x\in(a,b)}|f(x)| < \rho .
\]
Then
\begin{equation}\label{eq:Karl-1}
\sup_{x\in(a,b)}|f(x)|
\le
\rho + \sup_{[\alpha,\beta]\subset(a,b)}
\big(\var(f,[\alpha,\beta])\big)^{1/p}.
\end{equation}
\end{lemma}
\begin{proof}
There exists $x_0\in(a,b)$ such that $|f(x_0)|<\rho$. 
Consider an arbitrary $x\in(a,b)$. Let $I_x\subset(a,b)$ be the segment with
the endpoints $x$ and $x_0$. By \cite[Proposition~1.32(c)]{ABM14},
\[
|f(x)-f(x_0)|
\le 
\big(\var(f,I_x)\big)^{1/p}
\le 
\sup_{[\alpha,\beta]\subset(a,b)}
\big(\var(f,[\alpha,\beta])\big)^{1/p}.
\]
Hence
\begin{align*}
|f(x)|
\le
|f(x_0)|+\sup_{[\alpha,\beta]\subset(a,b)}
\big(\var(f,[\alpha,\beta])\big)^{1/p}
< 
\rho+\sup_{[\alpha,\beta]\subset(a,b)}
\big(\var(f,[\alpha,\beta])\big)^{1/p}.
\end{align*}
Since $x\in(a,b)$ is arbitrary, 
\[
\sup_{x\in(a,b)}|f(x)|
\le \rho +
\sup_{[\alpha,\beta]\subset(a,b)}
\big(\var(f,[\alpha,\beta])\big)^{1/p},
\]
which completes the proof.
\qed
\end{proof}

The next theorem says that an arbitrary pair of functions in 
$(BV_p[0,1])^2$ can be approximated by a 
pair of jointly nondegenerate  functions with the same product.
\begin{theorem}\label{th:nondegenerated}
Suppose that $1\le p<\infty$. For every $\eps>0$ and every pair 
of functions $(F,G)\in (BV_p[0,1])^2$ there is a pair of jointly 
nondegenerate  functions $(F_1,G_1)\in (BV_p[0,1])^2$ such that 
$F\cdot G=F_1\cdot G_1$ and
\[
\|F-F_1\|_{BV_p}<\eps,
\quad
\|G-G_1\|_{BV_p}<\eps.
\]
\end{theorem}
\begin{proof}
The idea of the proof is borrowed from the proof of \cite[Theorem~2.2]{KT19}.
Fix $\eps>0$. By Lemma~\ref{le:Shar}, we can find some $\delta>0$ such that
for every partition
\[
0\le x_1<x_2<\dots<x_m\le 1,
\]
we have
\begin{equation}\label{eq:nondeg-1}
|F(x_j)|<\delta\mbox{ for }
j\in\{1,\dots,m\}
\quad\Rightarrow\quad
\left(\sum_{j=1}^{m-1}|F(x_{j+1})-F(x_j)|^p\right)^{1/p}<\frac{\eps}{48}
\end{equation}
and
\begin{equation}\label{eq:nondeg-2}
|G(x_j)|<\delta\mbox{ for }
j\in\{1,\dots,m\}
\quad\Rightarrow\quad
\left(\sum_{j=1}^{m-1}|G(x_{j+1})-G(x_j)|^p\right)^{1/p}<\frac{\eps}{48}.
\end{equation}
Take 
\begin{equation}\label{eq:nondeg-3}
\eta:=\min\left\{\delta,\frac{\eps}{24},
\frac{1}{2}\sup_{x\in[0,1]}\big(|F(x)|+|G(x)|\big)\right\}.
\end{equation}
By the representation theorem for open sets on the real line (see, e.g.,
\cite[Theorem~3.11]{A74}), the interior of the set
$\{x\in[0,1]:|F(x)|+|G(x)|<\eta\}$ is the union of at most countable
collection of disjoint open intervals. Let $A_0$ be the collection of those
open intervals $U=(a,b)$, $a<b$, in this union such that
\[
\inf_{x\in U}\big(|F(x)|+|G(x)|\big) < \frac{\eta}{2}\, .
\]
We claim that there are only finitely many intervals in $A_0$. Indeed, assume
the contrary:
\[
A_0 =\big\{U_i=(a_i,b_i)\ :\ i\in\N,\ a_i<b_i\big\}.
\]
Without loss of generality, we can assume that $b_i\le a_{i+1}$ for every
$i\in\N$. Let $H:=|F|+|G|$. By the definition of the infimum, for every
$i\in\N$, there exists $x_i\in(a_i,b_i)$ such that $H(x_i)<\eta/2$.
On the other hand, there is at least one point $y_i$ such that
$b_i\le y_i\le a_{i+1}$ and $H(y_i)\ge\eta$. Hence
\[
\var(H,[0,1])
\ge 
\sum_{i=1}^\infty |H(y_i)-H(x_i)|^p
\ge
\sum_{i=1}^\infty\left(\eta-\frac{\eta}{2}\right)^p=+\infty,
\]
which is impossible since $H=|F|+|G|\in BV_p[0,1]$. Thus, for some $N\in\N$,
we have
\[
A_0=\big\{(a_1,b_1),\dots,(a_N,b_N)\big\}.
\]
Let 
\begin{equation}\label{eq:nondeg-4}
\rho := \min\left\{\frac{\eta}{2}\,, \frac{\eps}{48N}\right\}
\end{equation}
and let $A$ be the part of $A_0$ consisting of the intervals $(a_i,b_i)$ 
such that
\begin{equation}\label{eq:nondeg-5}
\inf_{x\in (a_i,b_i)}\big(|F(x)|+|G(x)|\big) < \rho .
\end{equation}
Relabelling $(a_i,b_i) \in A$ if necessary, we can assume
\[
A=\big\{(a_1,b_1),\dots,(a_n,b_n)\big\} ,
\]
where $n \le N$.

For $i\in\{1,\dots,n\}$, put
\begin{equation}\label{eq:nondeg-6}
c_i:=\max\left\{\sup_{x\in(a_i,b_i)}|F(x)|,\frac{\eps}{24n}\right\},
\quad
d_i:=\max\left\{\sup_{x\in(a_i,b_i)}|G(x)|,\frac{\eps}{24n}\right\}.
\end{equation}
It follows from definitions \eqref{eq:nondeg-6}, \eqref{eq:nondeg-3}
and the definition of the collection $A$ that 
\begin{equation}\label{eq:nondeg-6*}
\max_{1\le i\le n}\max\{c_i,d_i\}\le\frac{\eps}{24}.
\end{equation}
Taking into account the definition of the collection $A$
and \eqref{eq:nondeg-3},
we see that for every $i\in\{1,\dots,n\}$, every interval
$[\alpha,\beta]\subset (a_i,b_i)$ and every its partition
$\alpha=x_1<\dots<x_m=\beta$, one has $|F(x_j)|<\delta$
and $|G(x_j)|<\delta$ for $j\in\{1,\dots,m\}$.
Then \eqref{eq:nondeg-1}--\eqref{eq:nondeg-2} imply that
\begin{align}
&
\sum_{i=1}^n \sup_{[\alpha,\beta]\subset(a_i,b_i)}
\var(F,[\alpha,\beta])\le\left(\frac{\eps}{48}\right)^p,
\label{eq:nondeg-7}
\\
&
\sum_{i=1}^n \sup_{[\alpha,\beta]\subset(a_i,b_i)}
\var(G,[\alpha,\beta])\le\left(\frac{\eps}{48}\right)^p.
\label{eq:nondeg-8}
\end{align}
It follows from Lemma~\ref{le:Karl}, definition \eqref{eq:nondeg-4},
estimates \eqref{eq:nondeg-7}--\eqref{eq:nondeg-8}, and the inequality
\begin{equation}\label{eq:nondeg-9}
(t + \tau)^p \le 2^{p - 1} \left(t^p + \tau^p\right) , \quad t, \tau \ge 0
\end{equation} 
that
\begin{align}
\sum_{i=1}^n \left(\sup_{x\in(a_i,b_i)}|F(x)|\right)^p
&\le 
\sum_{i=1}^n \left(\rho + \sup_{[\alpha,\beta]\subset(a_i,b_i)}
(\var(F,[\alpha,\beta]))^{1/p}\right)^p 
\nonumber \\
&\le 
\sum_{i=1}^n 2^{p - 1} \left(\left(\frac{\eps}{48N}\right)^p + 
\sup_{[\alpha,\beta]\subset(a_i,b_i)}
\var(F,[\alpha,\beta])\right)
\le 
\left(\frac{\eps}{24}\right)^p,
\label{eq:nondeg-10}
\end{align}
and
$$
\sum_{i=1}^n \left(\sup_{x\in(a_i,b_i)}|G(x)|\right)^p
\le
\left(\frac{\eps}{24}\right)^p.
$$
Combining \eqref{eq:nondeg-6} and \eqref{eq:nondeg-10}, we see 
that
\begin{align}
\left(\sum_{i=1}^n c_i^p\right)^{1/p}
&\le 
\left(\sum_{i=1}^n \left(\sup_{x\in(a_i,b_i)}|F(x)|\right)^p
+
\sum_{i=1}^n\left(\frac{\eps}{24n}\right)^p
\right)^{1/p}
\nonumber\\
&\le 
\left(
\left(\frac{\eps}{24}\right)^p+n\left(\frac{\eps}{24n}\right)^p
\right)^{1/p}
\le 
\frac{\eps}{24}+\frac{\eps}{24}=\frac{\eps}{12}
\label{eq:nondeg-11}
\end{align}
and, similarly,
\begin{equation}\label{eq:nondeg-12}
\left(\sum_{i=1}^n d_i^p\right)^{1/p} \le \frac{\eps}{12}.
\end{equation}
Define $f,g:[0,1]\to\F\in\{\R,\C\}$ by
\begin{align}
f(x)&:=\left\{\begin{array}{lll}
F(x), & \displaystyle x\notin \bigcup_{i=1}^n (a_i,b_i),
\\[3mm]
c_i+d_i, & x\in(a_i,b_i), & i\in\{1,\dots,n\},
\end{array}\right.
\label{eq:nondeg-13}
\\
g(x)&:=\left\{\begin{array}{lll}
G(x), & \displaystyle x\notin \bigcup_{i=1}^n (a_i,b_i),
\\[3mm]
\displaystyle\frac{F(x)G(x)}{c_i+d_i}, 
& x\in(a_i,b_i), & i\in\{1\,\dots,n\}.
\end{array}\right.
\label{eq:nondeg-14}
\end{align}
It follows from \eqref{eq:nondeg-6}--\eqref{eq:nondeg-6*} 
and \eqref{eq:nondeg-13} that
\begin{equation}\label{eq:nondeg-15}
\|F-f\|_\infty 
=
\max_{1\le i\le n}\sup_{x\in(a_i,b_i)}|F(x)-(c_i+d_i)|
<
\max_{1\le i\le n}2(c_i+d_i)
\le 
2\left(\frac{\eps}{24}+\frac{\eps}{24}\right)=\frac{\eps}{6}
\end{equation}
and
\begin{align}
\var(F-f,[0,1])
\le & 
\sum_{i=1}^n \sup_{[\alpha,\beta]\subset(a_i,b_i)}
\var(F-(c_i+d_i),[\alpha,\beta])
\nonumber\\
&+
\sum_{i=1}^n \lim_{x\to a_i^+}|F(x)-(c_i+d_i)|^p
+
\sum_{i=1}^n \lim_{x\to b_i^-}|F(x)-(c_i+d_i)|^p
\nonumber\\
\le & 
\sum_{i=1}^n \sup_{[\alpha,\beta]\subset(a_i,b_i)}
\var(F,[\alpha,\beta])
+
2\sum_{i=1}^n\sup_{x\in(a_i,b_i)}\big(|F(x)|+c_i+d_i\big)^p
\nonumber\\
<&
\sum_{i=1}^n \sup_{[\alpha,\beta]\subset(a_i,b_i)}
\var(F,[\alpha,\beta])
+
4^p\sum_{i=1}^n (c_i+d_i)^p.
\label{eq:nondeg-16}
\end{align}
Combining \eqref{eq:nondeg-15}--\eqref{eq:nondeg-16} with
\eqref{eq:nondeg-7} and \eqref{eq:nondeg-11}--\eqref{eq:nondeg-12}, 
we see that
\begin{align}
\|F-f\|_{BV_p}
&=
\|F-f\|_\infty +\big(\var(F-f,[0,1])\big)^{1/p}
< 
\frac{\eps}{6}+\left(\left(\frac{\eps}{48}\right)^p
+
4^p\sum_{i=1}^n(c_i+d_i)^p\right)^{1/p}
\nonumber\\
&\le 
\frac{\eps}{6}+\frac{\eps}{48}
+4\left(\sum_{i=1}^n c_i^p\right)^{1/p}
+4\left(\sum_{i=1}^n d_i^p\right)^{1/p}
< 
\frac{\eps}{6}+\frac{\eps}{24}+\frac{\eps}{3}+\frac{\eps}{3}
=\frac{7\eps}{8}.
\label{eq:nondeg-17}
\end{align}
Analogously, it follows from \eqref{eq:nondeg-6}--\eqref{eq:nondeg-6*}
and \eqref{eq:nondeg-14} that
\begin{align}
\|G-g\|_\infty 
&=
\max_{1\le i \le n}\sup_{x\in(a_i,b_i)}
\left|G(x)-\frac{F(x)G(x)}{c_i+d_i}\right| 
\nonumber\\
&\le 
\max_{1\le i\le n}\left(
\sup_{x\in(a_i,b_i)}|G(x)|+
\sup_{x\in(a_i,b_i)}|G(x)|\sup_{x\in(a_i,b_i)}\frac{|F(x)|}{c_i+d_i}
\right)
\nonumber\\
&\le
\max_{1\le i\le n}\left(d_i+\frac{d_i\cdot c_i}{c_i+d_i}\right) 
<
2\max_{1\le i\le n} d_i
\le\frac{2\eps}{24}=\frac{\eps}{12}.
\label{eq:nondeg-18}
\end{align}
If $i\in\{1,\dots,n\}$ and $[\alpha,\beta]\subset(a_i,b_i)$,
then taking into account inequality \eqref{eq:nondeg-9}
and definitions \eqref{eq:nondeg-6}, we get
\begin{align}
&
\var\left(G\left(1-\frac{F}{c_i+d_i}\right),[\alpha,\beta]\right)
\nonumber\\
&\quad\le 
2^{p - 1} \left\{
\sup_{x\in[\alpha,\beta]}|G(x)|^p
\cdot
\var\left(1-\frac{F}{c_i+d_i},[\alpha,\beta]\right)
+
\var(G,[\alpha,\beta])
\cdot
\sup_{x\in[\alpha,\beta]}
\left|1-\frac{F(x)}{c_i+d_i}\right|^p
\right\}
\nonumber\\
&\quad\le 
2^p\left\{
\left(\frac{\displaystyle\sup_{x\in(a_i,b_i)}|G(x)|}{c_i+d_i}\right)^p\cdot
\var(F,[\alpha,\beta])
+
\var(G,[\alpha,\beta])\cdot 
\left(1+\frac{\displaystyle\sup_{x\in(a_i,b_i)}|F(x)|}{c_i+d_i}\right)^p
\right\}
\nonumber\\
&\quad\le
2^p\left\{
\left(\frac{d_i}{c_i+d_i}\right)^p\var(F,[\alpha,\beta])
+\var(G,[\alpha,\beta])\left(1+\frac{c_i}{c_i+d_i}\right)^p
\right\}
\nonumber\\
&\quad\le 
2^p\var(F,[\alpha,\beta])+4^p\var(G,[\alpha,\beta]).
\label{eq:nondeg-19}
\end{align}
Further, definitions \eqref{eq:nondeg-6} imply that for
$i\in\{1,\dots,n\}$,
\begin{align}
&
\lim_{x\to a_i^+}
\left|G(x)\left(1-\frac{F(x)}{c_i+d_i}\right)\right|^p
+
\lim_{x\to b_i^-}
\left|G(x)\left(1-\frac{F(x)}{c_i+d_i}\right)\right|^p
\nonumber\\
&\quad\le
2\sup_{x\in(a_i,b_i)}|G(x)|^p
\cdot 
\sup_{x\in(a_i,b_i)}\left|1-\frac{F(x)}{c_i+d_i}\right|^p
\le
2d_i^p\left(1+\frac{c_i}{c_i+d_i}\right)^p
\le
2^{p+1}d_i^p\le 4^pd_i^p.
\label{eq:nondeg-20}
\end{align}
It follows from \eqref{eq:nondeg-19}--\eqref{eq:nondeg-20} that
\begin{align}
\var(G-g,[0,1])
\le & 
\sum_{i=1}^n \sup_{[\alpha,\beta]\subset(a_i,b_i)} 
\var\left(G\left(1-\frac{F}{c_i+d_i}\right),[\alpha,\beta]\right) 
\nonumber\\
&+
\sum_{i=1}^n \lim_{x\to a_i^+}
\left|G(x)\left(1-\frac{F(x)}{c_i+d_i}\right)\right|^p
+
\sum_{i=1}^n \lim_{x\to b_i^-}
\left|G(x)\left(1-\frac{F(x)}{c_i+d_i}\right)\right|^p
\nonumber\\
\le & 
2^p \sum_{i=1}^n 
\sup_{[\alpha,\beta]\subset(a_i,b_i)}\var(F,[\alpha,\beta])
+
4^p\sum_{i=1}^n 
\sup_{[\alpha,\beta]\subset(a_i,b_i)}\var(G,[\alpha,\beta]) 
+
4^p\sum_{i=1}^n d_i^p.
\label{eq:nondeg-21}
\end{align}
Combining \eqref{eq:nondeg-18} and \eqref{eq:nondeg-21} with 
\eqref{eq:nondeg-7}--\eqref{eq:nondeg-8} and \eqref{eq:nondeg-12}, we see that
\begin{align}
\|G-g\|_{BV_p}
&=
\|G-g\|_\infty+\big(\var(G-g,[0,1])\big)^{1/p}
< 
\frac{\eps}{12}+\left(
2^p\left(\frac{\eps}{48}\right)^p
+
4^p\left(\frac{\eps}{48}\right)^p
+
4^p\sum_{i=1}^n d_i^p
\right)^{1/p}
\nonumber\\
&\le 
\frac{\eps}{12}+\frac{\eps}{24}+\frac{\eps}{12}+
4\left(\sum_{i=1}^n d_i^p\right)^{1/p}
<
\frac{\eps}{4}+\frac{\eps}{3}<\eps.
\label{eq:nondeg-22}
\end{align}
It follows from \eqref{eq:nondeg-17} and \eqref{eq:nondeg-22}
that $f,g\in BV_p[0,1]$, whence 
\[
h:=|f|+|g|\in BV_p[0,1].
\]
In view of Lemma~\ref{le:discontinuities}, the set $J$ of jumps of $h$
is at most countable. Let $\partial S$ and $\operatorname{int}(S)$ 
denote the boundary and the interior of a set 
$S\subset[0,1]$, respectively. Consider the sets
\[
S_\eta :=\{x\in[0,1]: h(x)<\eta\},
\quad
B_\eta:=\{x\in[0,1]: h(x)\ge\eta\}.
\]
Note that in view of the choice of $\eta$ in \eqref{eq:nondeg-3},
the set $B_\eta$ is nonempty. Then we have
$\partial(S_\eta)\setminus J\subset B_\eta$. Consider the set
\[
J_\eta:=\partial (S_\eta)\setminus B_\eta \subset J.
\]
We have
\begin{equation}\label{eq:nondeg-23}
[0,1]=B_\eta\cup S_\eta=B_\eta\cup\operatorname{int}(S_\eta)\cup J_\eta,
\end{equation}
where the sets $B_\eta$, $\operatorname{int}(S_\eta)$ and $J_\eta$ are 
pairwise disjoint. 

We claim that the set
\[
J_\eta^s:=\{y\in J_\eta: h(y)<\eta/2\}
\]
is finite. Indeed, since $J_\eta^s\subset J_\eta\subset J$, the set
$J_\eta^s$ is at most countable. Assume the contrary, that is, that the
set $J_\eta ^s$ is infinite. Let $J_\eta^s=\{y_j\}_{j=1}^\infty$
and $y_j<y_{j+1}$ for all $j\in\N$. Then for every $j\in\N$,  there
exists $x_j\in B_\eta$ such that $y_{2j-1}<x_j<y_{2j+1}$. Therefore
\[
\var(h,[0,1])\ge \sum_{j=1}^\infty |h(x_j)-h(y_{2j-1})|^p
\ge \sum_{j=1}^\infty\left(\eta-\frac{\eta}{2}\right)^p=+\infty,
\]
which is impossible since $h\in BV_p[0,1]$. Thus, the set $J_\eta^s$
is finite.

Consider the (obviously, finite)  set
\[
J_\eta^0:=\{y\in J_\eta^s:h(y)=0\}.
\]
Let $k$ be the cardinality of $J_\eta^0$. Define the functions 
$F_1,G_1:[0,1]\to\F\in\{\R,\C\}$ by
\begin{equation}\label{eq:nondeg-24}
F_1(x):=\left\{\begin{array}{ll}
f(x), & x\in[0,1]\setminus J_\eta^0,
\\[3mm]
\displaystyle\frac{\eps}{24 k}, & x\in J_\eta^0,
\end{array}\right.
\end{equation}
and
\begin{equation}\label{eq:nondeg-25}
G_1(x) :=g(x), \quad x\in[0,1].
\end{equation} 
It is clear that 
\begin{equation}\label{eq:nondeg-26}
f(x)=g(x)=0,
\quad
x\in J_\eta^0.
\end{equation}
It follows from \eqref{eq:nondeg-13}--\eqref{eq:nondeg-14}
and \eqref{eq:nondeg-24}--\eqref{eq:nondeg-26} that
\begin{equation}\label{eq:nondeg-27}
F(x)G(x)=f(x)g(x)=F_1(x)G_1(x),
\quad x\in[0,1].
\end{equation}
Moreover,
\begin{align}
\|F_1-f\|_{BV_p}
=
\|F_1-f\|_\infty + \big(\var(F_1-f,[0,1]\big)^{1/p}
=
\frac{\eps}{24k}+\left(2k\left(\frac{\eps}{24k}\right)^p\right)^{1/p}
\le\frac{2k+1}{24k}\eps
\le\frac{\eps}{8}.
\label{eq:nondeg-28}
\end{align}
Combining \eqref{eq:nondeg-17} and \eqref{eq:nondeg-28},
we arrive at the following:
\begin{equation}\label{eq:nondeg-29}
\|F-F_1\|_{BV_p}
\le 
\| F-f\|_{BV_p}+\|f-F_1\|_{BV_p}
<\frac{7\eps}{8}+\frac{\eps}{8}=\eps.
\end{equation}
In view of \eqref{eq:nondeg-22} and \eqref{eq:nondeg-25},
we have
\begin{equation}\label{eq:nondeg-30}
\|G-G_1\|_{BV_p}=\|G-g\|_{BV_p}<\eps.
\end{equation}
For a set $S\subset[0,1]$, let
\[
I(S):=\inf_{x\in S}\big(|F_1(x)|+|G_1(x)|\big).
\]
Then it follows from \eqref{eq:nondeg-24}--\eqref{eq:nondeg-26}
that
\begin{align}
I_1
&:=
I(B_\eta)=\inf_{x\in B_\eta}\big(|f(x)|+|g(x)|\big)\ge\eta>0,
\label{eq:nondeg-31}
\\
I_2
&:=I(J_\eta\setminus J_\eta^s) 
=
\inf_{y\in J_\eta\setminus J_\eta^s}
\big(|f(y)|+|g(y)|\big) \ge \frac{\eta}{2} > 0 ,
\label{eq:nondeg-32}
\\
I_3
&:=I(J_\eta^s\setminus J_\eta^0) 
=
\min_{y\in J_\eta^s\setminus J_\eta^0}\big(|f(y)|+|g(y)|\big)>0
\label{eq:nondeg-33}
\end{align}
(recall that the set $J_\eta^s\setminus J_\eta^0$ is finite), and
\begin{equation}\label{eq:nondeg-34}
I_4:=I(J_\eta^0)\ge\frac{\eps}{24k}>0.
\end{equation}
By the definition of the collection $A$ and definitions
\eqref{eq:nondeg-13}--\eqref{eq:nondeg-14}
and \eqref{eq:nondeg-24}--\eqref{eq:nondeg-25},
we have
\begin{equation}\label{eq:nondeg-35}
I_5
:=
I\left(
\operatorname{int}(S_\eta)\setminus\left(\bigcup_{i=1}^n (a_i,b_i)\right)
\right)
=
\inf_{x\in 
\operatorname{int}(S_\eta)\setminus\left(\bigcup_{i=1}^n (a_i,b_i)\right)}
\big(|F(x)|+|G(x)|\big) \ge \rho >0
\end{equation}
(see \eqref{eq:nondeg-4} and \eqref{eq:nondeg-5})
and, in view of \eqref{eq:nondeg-6}, we see that
\begin{equation}\label{eq:nondeg-36}
I_6 
:=
I\left(\bigcup_{i=1}^n (a_i,b_i)\right)
\ge 
\min_{1\le i\le n}\inf_{x\in(a_i,b_i)}\big(|f(x)|+|g(x)|\big)
\ge 
\min_{1\le i\le n}(c_i+d_i)\ge\frac{\eps}{12n}>0.
\end{equation}
It follows from \eqref{eq:nondeg-23} and
\eqref{eq:nondeg-31}--\eqref{eq:nondeg-36} that
\[
I([0,1])\ge\min_{1\le j\le 6} I_j>0.
\]
Thus, functions $F_1,G_1\in BV_p[0,1]$ are jointly nondegenerate. 
Combining this observation with \eqref{eq:nondeg-27}
and \eqref{eq:nondeg-29}--\eqref{eq:nondeg-30},
we arrive at the conclusion of the theorem.
\qed
\end{proof}
\section{Proof of the main result and final remarks}
\label{sec:proof-conjecture}
\subsection*{Proof of Theorem~\ref{th:main}}
Take an arbitrary pair $(F,G)\in (BV_p[0,1])^2$. Fix $\eps>0$. It follows
from Theorem~\ref{th:nondegenerated} that there exists a 
pair of jointly nondegenerate functions $(F_1,G_1)\in(BV_p[0,1])^2$ 
such that
\begin{equation}\label{eq:proof-main-1}
F\cdot G=F_1\cdot G_1
\end{equation}
and
\begin{equation}\label{eq:proof-main-2}
\|F-F_1\|_{BV_p}<\eps/2,
\quad
\|G-G_1\|_{BV_p}<\eps/2.
\end{equation}
By Corollary~\ref{co:local-BVp}, there exists a $\delta>0$ such that
\begin{equation}\label{eq:proof-main-3}
B_{BV_p[0,1]}(F_1\cdot G_1,\delta)
\subset 
B_{BV_p[0,1]}(F_1,\eps/2)
\cdot 
B_{BV_p[0,1]}(G_1,\eps/2) .
\end{equation}
Combining \eqref{eq:proof-main-1}--\eqref{eq:proof-main-3}, we arrive at
the following:
\[
B_{BV_p[0,1]}(F\cdot G,\delta)
\subset 
B_{BV_p[0,1]}(F_1,\eps/2)
\cdot 
B_{BV_p[0,1]}(G_1,\eps/2)
\subset
B_{BV_p[0,1]}(F,\eps)
\cdot 
B_{BV_p[0,1]}(G,\eps).
\]
Thus, the multiplication in the Banach algebra $BV_p[0,1]$  is locally open
at the pair $(F,G)$. Since $(F,G)\in (BV_p[0,1])^2$ is an arbitrary pair,
we conclude that the multiplication in $BV_p[0,1]$ is an open bilinear mapping.
\qed

Let $1\le p<\infty$ and $\Lambda_p BV[0,1]$ be the Banach algebra of 
all functions of bounded variation in the Shiba-Waterman sense.
We conclude the paper with the following.
\begin{conjecture}
The multiplication in the Banach algebra $\Lambda_p BV[0,1]$ is an open 
bilinear mapping.
\end{conjecture}

In view of Corollary~\ref{co:local-Shiba-Waterman}, to confirm this conjecture,
one has to prove that every pair of functions $(f,g)\in(\Lambda_pBV[0,1])^2$
can be approximated by a pair of jointly nondegenerate functions 
$(f_1,g_1)\in(\Lambda_pBV[0,1])^2$ such that $f\cdot g=f_1\cdot g_1$.

\section*{Acknowledgments.} 
This work was partially supported by the Funda\c{c}\~ao para a Ci\^encia e a
Tecnologia (Portu\-guese Foundation for Science and Technology)
through the project
UIDB/MAT/00297/2020 (Centro de Matem\'atica e Aplica\c{c}\~oes).


\end{document}